\documentclass[11pt]{amsart}

\pdfoutput=1

\usepackage[utf8]{inputenc}
\usepackage[T1]{fontenc}
\usepackage{lmodern}
\usepackage{amsthm, amssymb, amsmath, amsfonts, mathrsfs}

\usepackage{microtype}

\usepackage[pagebackref,colorlinks=true,pdfpagemode=none,urlcolor=blue,
linkcolor=blue,citecolor=blue]{hyperref}

\usepackage{amsmath,amsfonts,amssymb,amsthm}
\usepackage{amsthm, amssymb, amsmath, amsfonts, mathrsfs}

\usepackage{mathrsfs}
\usepackage{MnSymbol}
\usepackage{scalerel} 

\usepackage{color}
\usepackage{accents}

\usepackage{bbm}

\definecolor{labelkey}{gray}{.8}
\definecolor{refkey}{gray}{.8}

\definecolor{darkred}{rgb}{0.9,0.1,0.1}
\definecolor{darkgreen}{rgb}{0,0.5,0}

\newtheorem{theorem}{Theorem}[section]
\newtheorem{lemma}[theorem]{Lemma}

\newtheorem{corollary}[theorem]{Corollary}

\theoremstyle{remark}

\renewenvironment{proof}[1][Proof]{ {\itshape \noindent {#1.}} }{$\Box$
\medskip}

\numberwithin{equation}{section}
\newcommand{\R}{\mathbb{R}}

\newcommand{\PP}{\mathbf{P}}

\newcommand{\e}{\varepsilon}

\newcommand{\F}{\mathcal{F}}

\newcommand{\eps}{\varepsilon}

\newcommand{\Var}{\mathrm{Var}}

\newcommand{\EE}{\mathbf{E}}

\newcommand{\cov}{\mathrm{Cov}}
\newcommand{\corr}{\mathrm{Corr}}
\newcommand{\cP}{\mathcal{P}}

\renewcommand{\d}{\mathrm{d}}
\newcommand{\hk}{\mathsf{p}}
\newcommand{\calW}{\mathcal{W}}
\newcommand{\bartau}{\bar{\tau}}
\newcommand{\bartheta}{\bar{\theta}}
\newcommand{\ip}[1]{\langle#1\rangle}
\newcommand{\Ip}[1]{\big\langle#1\big\rangle}

\newcommand{\jfn}{\mathsf{j}}
\newcommand{\she}{\scriptscriptstyle\mathrm{she}}
\newcommand{\sg}{\mathcal{Q}}
\newcommand{\intv}[1]{#1}

\begin{document}

\title[SHF is a black noise]{Stochastic Heat Flow is a Black Noise}
\author{Yu Gu, Li-Cheng Tsai}

\address[Yu Gu]{Department of Mathematics, University of Maryland, College Park, MD 20742, USA}

\address[Li-Cheng Tsai]{Department of Mathematics, University of Utah, Salt Lake City, UT 84112, USA}

\begin{abstract}
We prove that the stochastic heat flow \cite{caravenna2023critical,tsai2024stochastic} is a black noise in the sense of Tsirelson \cite{Tsirelson2004}.
As a corollary, the 2$d$ stochastic heat equation driven by a mollified spacetime white noise becomes asymptotically independent of that noise in the scaling limit.

\bigskip



\end{abstract}
\maketitle

\section{Introduction}

\subsection{SHF formulation}
In this paper, we prove that the Stochastic Heat Flow (SHF) is a black noise, working with the SHF formulation in \cite{tsai2024stochastic}, as a continuous process characterized by a list of axioms. The SHF was first constructed in \cite{caravenna2023critical} as the finite dimensional distribution limit of discrete polymers in random environments. The methods of this paper require working with the formulation of \cite{tsai2024stochastic}. First, the continuity of the SHF ensures a sufficiently good grasp of the underlying sigma algebra. Second, and more subtly, our proof crucially relies on working with a formulation that does not explicitly reference pre-limiting models.

\subsection{Main result}
By \cite{tsai2024stochastic}, the SHF $Z^{\theta}=Z=(Z_{s,t})_{s\leq t}$ is a continuous process indexed by $(s,t)\in \R^2_{\leq}:=\{(s,t)\in \R^2:s\leq t\}$ that takes values in $M_+(\R^4)$, the space of positive locally finite measures on $\R^4$. 
The law of $Z$ is characterized by a few axioms, which we will recall in Section~\ref{s.pf.shf}.
Throughout the paper, we will fix the parameter $\theta\in\R$ and omit the dependence of $Z$ on $\theta$.

Define the sample space as
\[
\Omega=C(\R^2_{\leq}, M_+(\R^4)).
\]
For any $f\in C_c(\R^4)$,  $\mu\in \Omega$ and $(s,t)\in \R^2_{\leq}$, write $\mu_{s,t}f=\int_{\R^4}\mu_{s,t}(\d x, \d x')f(x,x')$, with $x,x'\in\R^2$. 
Let $\F$ be the $\sigma$-algebra generated by sets of the form 
\begin{align}\label{e.filt}
    \{\mu\in \Omega: \mu_{s,t}f\in (a,b)\}, \quad\quad \mbox{ with } f\in C_c(\R^4),\, (s,t)\in \R^2_{\leq}, \,a<b.
\end{align}
Similarly, for any $s<t$, let $\F_{s,t}$ be the $\sigma$-algebra generated by sets of the form 
\[
\{\mu\in \Omega: \mu_{s',t'}f\in (a,b)\}, \quad\quad \mbox{ with } f\in C_c(\R^4), \,[s',t']\in [s,t],\, a<b.
\]
For any $u\in \R$, define the shift operator $T_u: \Omega\to\Omega$ by 
\[
(T_u\mu)_{s,t}=\mu_{s+u,t+u}.
\]
Let $\PP$ be the probability measure on $(\Omega,\F)$ induced by $Z=(Z_{s,t})_{s\leq t}$. According to \cite[Section 3d1]{Tsirelson2004}, the probability space $(\Omega,\F,\PP)$, together with the sub-$\sigma$-algebras $\{\F_{s,t}\}_{(s,t)\in\R^2_{\leq}}$ and the shift operator $(T_u)_{u\in\R}$, forms a \emph{noise}, or \emph{a homogeneous continuous product of probability
space}.

We now define the space of linear random variables on $(\Omega,\F,\PP)$ as in \cite[6a1~Proposition]{Tsirelson2004}:
\[
H_1=\{X\in L^2(\Omega): \EE[X|\F_{s,t}]=\EE[X|\F_{s,u}]+\EE [X|\F_{u,t}] \mbox{ for all } s<u<t\}.
\]
It is clear that $H_1$ is a closed subspace of $L^2(\Omega)$. 
According to \cite[Section~7a1]{Tsirelson2004}, the noise is black if and only if $H_1=\{0\}$.
The main result of the paper is as follows.

\begin{theorem}\label{t.mainth}
The noise defined above is black, namely $H_1=\{0\}$. 
\end{theorem}

An important consequence of Theorem~\ref{t.mainth} is an asymptotic independence result, stated and proved as Corollary~\ref{c.indep} below.
Independently of the present work, the work \cite{CD25} proved an enhanced noise sensitivity result for  discrete polymers in random environments and achieved an asymptotic independence of the polymers and the SHF, analogous to Corollary~\ref{c.indep}
\footnote{One of us (YG) attended a presentation by Francesco Caravenna in February 2025 on this result.}.
%

To state this precisely, we recall the convergence of the noise-mollified stochastic heat equation to the SHF obtained in \cite[Section~4]{tsai2024stochastic}.
For any $(s,x')\in \R^3$, let $\{Z^\eps_{s,t}(x',x)\}_{t\geq s, x\in\R^2}$ be the solution to 
\begin{align}
	\label{e.she}
	&\partial_t Z^\eps=\tfrac12\Delta_x Z^\eps+\sqrt{\beta^\eps} Z^\eps \xi^\eps(t,x), \quad\quad  t>s, x \in\R^2, 
\\
	\notag
	&Z^\eps_{s,s}(x',x)=\delta(x-x').
\end{align}
Here $\xi^\eps$ is the spatial mollification of the spacetime white noise $\xi$: 
\begin{align}
	\label{e.noise.mollified}
	\xi^\eps(t,x)=\int_{\R^2} \d x''\,\varphi^\eps(x-x'')\,\xi(t,x''),
\end{align}
with mollifier $\varphi^\eps(x)=\tfrac{1}{\eps^2}\varphi(\tfrac{x}{\eps})$ where $\varphi\in C_c^\infty(\R^2)$ is symmetric and satisfies $\int_{\R^2} \varphi =1$. 
Define $\Phi=\varphi\star \varphi$.
The coupling constant $\beta^\eps$ is given by 
\begin{align*}
	\beta^\eps=\frac{2\pi}{|\log \eps|}+\frac{\pi}{|\log \eps|^2}\Big(\theta-2\log 2+2\gamma+2\int_{\R^4} \d x\,\d x'\,\Phi(x)\,\log |x-x'|\,\Phi(x')\Big),
\end{align*}
where $\gamma= 0.577...$ is the Euler–Mascheroni constant, and $\theta\in\R$ is the fixed parameter mentioned at the beginning of this paper.
With slight abuse of notation, we view $Z^\eps$ as an $\Omega$-valued random variable through
\[
Z^\eps_{s,t} f:=\int_{\R^4} \d x'\,\d x \, Z^\eps_{s,t}(x',x)\, f(x',x), \quad\quad s\leq t,  f\in C_c(\R^4).
\]
Equip $M_+(\R^4)$ with the vague topology: the topology induced by $C_c(\R^4)$.
It was shown in \cite[Section~4]{tsai2024stochastic} that $Z^\eps\Rightarrow Z$ in law as $\eps\to0$, under the uniform-on-compact topology of $C(\R^2_{\leq}, M_+(\R^4))$. 
Since $Z^\eps$ is a measurable function of $\xi$, it is natural to investigate the dependence of $Z^\eps$ on $\xi$ as $\eps\to0$. 
The following corollary asserts that they become independent in the limit.

\begin{corollary}\label{c.indep}
We have that $(Z^\eps,\xi^\eps)\Rightarrow (Z,\xi)$ in law as $\eps\to0$, under the product topology of $C(\R^2_{\leq}, M_+(\R^4))\times \mathcal{S}'(\R^2)$, where $Z$ is independent of $\xi$.
\end{corollary}
\begin{proof}
The key property is that any coupling of a black and a white noise must make them independent.
This is proven in \cite[Theorem~1.11]{himwich2024directed} when the black noise is the directed landscape.
The same proof applies here and shows that $Z$ and $\xi$ are independent under any coupling of the two.
Next, since $Z^\eps\Rightarrow Z$ and $\xi^\eps\Rightarrow \xi$, $(Z^\eps,\xi^\eps)$ is tight.
Any limit point of the law of $(Z^\eps,\xi^\eps)$ has its first marginal equal to the law of $Z$ and its second marginal equal to the law of $\xi$.
The two marginals must be independent, so the claim follows. 
\end{proof}

\subsection{Discussion on noise sensitivity}\label{s.noisesensitivity}
We emphasize that, in the present context, our main result goes beyond noise sensitivity. Roughly speaking, noise sensitivity concerns how a pre-limiting model responds to perturbations of the underlying noise---specifically, to partial resampling of that noise; see \cite{garban2015noise} for an introduction.
Below, we demonstrate that $Z^\eps$ is noise sensitive and determine the scale of its sensitivity. In addition, we provide an example illustrating that, in this setting, noise sensitivity as defined below is strictly weaker than the black noise or asymptotic independence property.
%
%

Recall that a family of $\xi$-measurable random variables $\{F^\eps(\xi)\}_{\eps>0}$ is noise sensitive if, for every $\tau>0$,
\begin{equation}\label{e.decayco}
	C(\tau,F^\e)
	:=
	\corr\big( F^\eps(\xi), F^\eps(e^{-\tau}\xi+\sqrt{1-e^{-2\tau}}\xi') \big) \longrightarrow 0,
	\quad
	\text{as } \eps\to 0,
\end{equation}
where $\xi'$ denotes an independent copy of $\xi$.

To put Theorem~\ref{t.mainth} and Corollary~\ref{c.indep} into perspective, let us examine the noise sensitivity of $Z^\e$.
Take nonnnegative $g,g'\in C_c(\R^2)$ that are not identically zero, view $Z^\eps_{0,1}=Z^{\eps}_{0,1}(\xi)$ as a measurable function of $\xi$, and consider
\begin{align*}
	F_{\she}^{\e}(\xi)=Z^{\e}_{0,1}\, g\otimes g' - \EE Z^{\e}_{0,1}\, g\otimes g'\ .
\end{align*}
The Fourier spectrum of $F_{\she}^{\e}$ ``escapes'' to infinite.
More precisely, let $F_k$ denote the $k$th Wiener chaos of $F\in L^2(\Omega)$ and let $c_k(F):=(\EE F_k^2)^{1/2}$, which is sometimes called the Fourier coefficient.
Indeed,
\begin{align}
    \label{e.chaos.2nd}
    \sum_{k\geq 1} c_k(F_{\she}^{\e})^2
    =
    \EE (F_{\she}^{\e})^2,
\end{align}
and by \cite[Theorem~3.2]{bertini1998two}, the rhs converges to a positive limit as $\e\to 0$.
On the other hand, using \eqref{e.she.chaos}, one can show that $c_k(F_{\she}^{\e})^2\to 0$ as $\e\to 0$, for each fixed $k$.
Namely, the weight in the sum in \eqref{e.chaos.2nd} escapes to ever higher $k$s as $\e\to 0$.
This property implies that $F_{\she}^{\e}$ is noise sensitive.
(By \eqref{e.cov} for example; see \cite{benjamini1999noise} and \cite[Proposition~IV.2]{garban2015noise} for a similar statement in the Boolean context.)

A more careful Fourier analysis gives the \emph{scale} of the noise sensitivity.
Call $F^\e$ noise sensitive at scale $s_\eps\to 0$ if $C(\tau_\e,F^\e)\to 1$ whenever $\tau_\e\ll s_\e$ and $C(\tau_\e,F^\e)\to 0$ whenever $\tau_\e\gg s_\e$, where $a_\eps\ll b_\eps$ means that $a_\eps/b_\eps\to 0$.
We show in Appendix~\ref{s.noise.sensitive} that, for every $\bartau \geq 0$,
\begin{align}
	\label{e.noise.sensitive}
	C(\tfrac{\bar{\tau}}{2|\log\e|}, F_{\she}^{\e})
	\longrightarrow
	\frac{
	\ip{ g^{\otimes 2}, \calW^{\theta-\bartau}(1) \, g'{}^{\otimes 2} }
	}{  
	\ip{ g^{\otimes 2}, \calW^{\theta}(1) \, g'{}^{\otimes 2} }
	}.
\end{align}
Hereafter, $\ip{\cdot,\cdot}$ denotes the inner product on $L^2(\R^d)=L^2(\R^d,\d x)$, $\hk(t,x)$ denotes the 2$d$ heat kernel, and $\calW^{\bartheta}(t)$ denotes the bounded integral operator on $L^2(\R^4)$ with the kernel
\begin{align}
\label{e.calW}
\begin{split}
	\calW^{\bartheta}(t,x,x')
	=&
	\int_{u+u'+u''=t} \d u \d u'
	\int_{\R^2} \d y_c \,
	\prod_{i=1}^2 \hk(u,x_i-y)
\\
	&\times  
	\int_{\R^2} \d y'\, 
	4\pi\,\jfn^{\bartheta}(u')\hk(\tfrac{u'}{2},y-y') 
	\prod_{i=1}^2 \hk(u'',y'-x_i'),
\end{split}
\end{align}
where $x=(x_1,x_2),x'=(x'_1,x'_2)\in\R^4$ and
\begin{align}
    \label{e.jfn}
    \jfn^{\bartheta}(t) = \int_0^\infty \d u\, \frac{ u^{t-1} e^{\bartheta t}}{\Gamma(u)}\ .
\end{align}
It is straightforward to check that the rhs of \eqref{e.noise.sensitive} tends respectively to $1,0$ as $\bartau\to 0,\infty$.
Also, $C(\tau,F)$ is decreasing in $\tau$ (see \eqref{e.cov} for example).
Hence, $F_{\she}^{\e}$ is noise sensitive at scale $s_\e=|\log\e|^{-1}$.

The above noise-sensitivity results and Fourier analysis do not imply our main result. In general, the notion of noise sensitivity---interpreted as the decay of correlation in \eqref{e.decayco}---implies neither the black noise property nor the asymptotic independence between the limiting object and the noise used in its construction. Here is an example: let $\xi=B'$ where $B$ is a standard Brownian motion. By \cite[Theorem 1]{HN09} and \cite[Theorems~1.2]{deya2015L2modulus}, there exists a family of $\xi$-measurable random variables $F^\e$ such that $\EE F^{\e}=0$, $\EE(F^\e)^2$ converges to a finite positive limit, and $c_k(F^\e)\to 0$ for every $k$. However, the pair $(F^\e,\xi)$ converges in law to $(F,\xi)$ in which $F$ and $\xi$ remain dependent. In fact, $F$ is constructed from an independent Gaussian random variable and $\xi$, and notably, $\xi$ appears in $F$ through the local time of the Brownian motion $B$.

This example shows that Theorem~\ref{t.mainth} cannot be derived solely from a Fourier analysis of the pre-limiting models.
Our proof, on the other hand, operates at the level of the limiting process $Z$ with the aid of the axioms developed in \cite{tsai2024stochastic}.

\subsection{Some related literature}

The SHF is particularly interesting because it is one of the rare examples where the scaling limit of a critical stochastic PDE can be rigorously obtained and analyzed. 
Moreover, it serves as a natural generalization of the 1$d$ stochastic heat equation and may shed light on the largely unexplored $1+2d$ Kardar--Parisi--Zhang universality class.
The first indication that an intriguing limit of \eqref{e.she} should exist comes from \cite{bertini1998two} through the study of the second moment. The moments are closely related to the two-dimensional delta-Bose gas and more generally, to Schr\"{o}dinger operators with point interactions, which have been extensively studied in the mathematical physics and functional analysis literature; we refer to \cite{Albeverio1988,DellAntonio1994,Dimock2004} and the references therein. Higher moments are subsequentially obtained in \cite{Caravenna2019,gu2021moments} and further studied in \cite{Chen2024}. Moment results like these ensure the tightness of the pre-limiting models, but the uniqueness of the limit point remained open, until the work \cite{caravenna2023critical} proved that the finite-dimensional distributions of the discrete polymers in random environments converge. Recently, the work \cite{tsai2024stochastic} established the axiomatic, continuous-process formulation and construction of the SHF used in the current paper.
Some other properties related to the SHF are studied in \cite{Caravenna2023,Clark2024,liu2024moments,caravenna2025singularity,nakashima2025martingale,chen2025martingale}.



Introduced in \cite{Tsirelson1998,Tsirelson2004}, the notion of black noise offers a way to identify intriguing stochastic objects that are not amenable to traditional stochastic analysis.
Known examples of black noises include Arratia's coalescing flow, the sticky flows, and the Brownian web \cite{tsirelson2004lectures,lejan2004sticky,ellis2016brownian}, the critical planner percolation \cite{shramm2011scaling}, and the directed landscape \cite{himwich2024directed}.

Closely related to our work is the recent paper \cite{himwich2024directed}, in which the authors studied the directed landscape---one of the universal objects in the $1+1d$ KPZ universality class---and proved that it is a black noise. 
This is achieved by establishing variance bounds on polynomial observables (Equation~(2.2) therein).
For the SHF, polynomial observables likely do not suffice. 
The moments of the SHF are conjectured to grow very rapidly (\cite{Rajeev1999} and \cite[Remark~1.8]{gu2021moments}),  in which case polynomial observables may fail to approximate more general observables.
Our proof takes a different route and directly tackles non-polynomial observables.

\subsection*{Acknowledgement}
We learned about the counterexample discussed in Section~\ref{s.noisesensitivity} from David Nualart.
We thank Jeremy Quastel and Mo Dick Wong for useful discussions. 
YG was partially supported by the NSF through DMS-2203014.
LCT was partially supported by the NSF through DMS-2243112 and the Alfred P.\ Sloan Foundation through the Sloan Research Fellowship FG-2022-19308.

\section{Proof of Theorem~\ref{t.mainth}}
\subsection{Reduction to a variance bound}

To prove Theorem~\ref{t.mainth}, we use the criterion in \cite[Lemma~7a2, Corollary~7a3]{Tsirelson2004}, slightly restated below:

\textit{Tsirelson's criterion}: Let $\cP:L^2(\Omega)\to H_1$ be the projection operator.  Fix a countable dense subset $A$ of $\R$, let $\{A_n\}_{n\geq1}$ be an increasing sequence of finite subsets such that $A=\cup_n A_n$. Write $A_n=\{-\infty<t_1<\ldots<t_{|A_n|}<\infty\}$ and interpret $A_n$ as a partition of $\R$. For any $X\in L^2(\Omega)$, define the projection
\begin{equation}\label{e.defPn}
\cP_nX=\sum_{j=1}^{|A_n|+1}\big(\EE[X|\F_{t_{j-1},t_j}]-\EE[X]\big) ,
\end{equation}
with the convention that $t_0=-\infty,t_{|A_n|+1}=\infty$. 
Then $\cP_nX\to \cP X$ in $L^2(\Omega)$ as $n\to\infty$. 

By the criterion, to prove $H_1=\{0\}$, it suffices to find a dense subset $S\subset L^2(\Omega)$ and show that for all $X\in S$
\[
\EE |\cP_nX|^2=\sum_{j=1}^{|A_n|+1}\Var\big[\EE[X|\F_{t_{j-1},t_j}]\big]\longrightarrow 0, \quad \mbox{ as } n\to\infty.
\]

We begin by identifying the set $S$:
\begin{lemma}
The set 
\begin{align}\label{e.defS}
\begin{split}
    S = \Big\{ 
        h(Z_{s_1,t_1} &g_1\otimes g'_1, \ldots, Z_{s_N,t_N} g_N\otimes g'_N)
        \,:\,
        N\geq 1, h\in C_c^\infty(\R^N), 
\\
    &
        s_i<t_i \in\R, \ g_i,g'_i\in C_c(\R^2) 
        \text{ for all }i=1,\ldots,N
    \Big\}
\end{split}
\end{align}
is dense in $L^2(\Omega)$.
\end{lemma}

The prime in $g'_i$ denotes an alternative and \emph{not the derivative}. 

\begin{proof}
First, $\F$ is generated by countably many random variables of the form $Z_{s_i,t_i}\,g_i\otimes g'_i$ where $s_i<t_i$ and $g_i,g'_i\in C_c(\R^2)$.
To see why, note that we can restrict the $s\leq t$ in \eqref{e.filt} to $s< t\in\mathbb{Q}$ because the sample space consists of continuous $\mu$ and because $Z_{s,s}$ is deterministic (by Axioms~\eqref{d.shf.} and \eqref{d.shf.mome} in Section~\ref{s.pf.shf}).
Further, letting $C_r(\R^d)$ denote the space of continuous functions on $\R^d$ supported in $\{x:|x|<r\}$ and taking any countable $D\subset C_c(\R^2)$ such that $D^{\otimes 2}\cap C_r(\R^4)$ is dense in $C_r(\R^4)$ for every $r<\infty$, we can restrict the $f\in C_c(\R^4)$ in \eqref{e.filt} to $f\in D^{\otimes 2}$.

Next, by \cite[Lemma 2.3]{himwich2024directed}, if we take bounded $h:\R^N\to \R$, then the set of random variables $h(Y)$ with $Y=(Z_{s_1,t_1} g_1\otimes g'_1,\ldots,Z_{s_N,t_N}g_N\otimes g'_N)$ is dense in $L^2(\Omega)$. It remains to fix such an $h(Y)$ and show that there exists a sequence of $h_k \in C_c^\infty(\R^N)$ such that  $\EE[ |h_k(Y)-h(Y)|^2]\to0$  as $k\to\infty$. By \cite[Theorem 3.14]{rudin1987real}, we may assume that $h\in C_c(\R^N)$, since $C_c(\R^N)$ is dense in $L^2(\R^N,\PP_Y)$ where $\PP_Y$ is the probability measure on $(\R^N,\mathcal{B}(\R^N))$ induced by $Y$. We now define $h_k:=\phi_k\star h$ where $\phi_k$ is a standard mollifier (an approximation of identity). The proof is complete by invoking the bounded convergence theorem.
\end{proof}

Fix an element in $S$:
\begin{equation}\label{e.defX}
X=h(Z_{s_1,t_1} g_1\otimes g'_1,\ldots,Z_{s_N,t_N}g_N\otimes g'_N).
\end{equation}
Our goal is to show that $\cP_n X\to0$ in $L^2(\Omega)$ as $n\to\infty$. 
Given the set $\{(s_i,t_i)\}_{i=1}^N$, we choose the sequence $A_n$ so that every interval $[s,t]$ in the partition either satisfies $[s,t]\subset [s_i,t_i]$ or satisfies $t=s_i$ or $s=t_i$, i.e., overlapping with $[s_i,t_i]$ only at a single point.

Let $\underline{s}=\min\{s_1,\ldots,s_N\}$, $\overline{t}=\max\{t_1,\ldots,t_N\}$. By the above discussion, it is enough to show that for any such $[s,t]\in [\underline{s},\overline{t}]$, we have  
\[
\Var [\EE[X|\F_{s,t}]]\ll t-s, \quad \mbox{ as }t-s\to 0.
\] In fact, we will prove the following bound  
\begin{equation}\label{e.keyestimate}
\Var [\EE[X|\F_{s,t}]]\leq C\frac{t-s}{|\log (t-s)|}, \quad\quad \mbox{ if } t-s\leq \tfrac12.
\end{equation}

\subsection{Properties of the SHF}
\label{s.pf.shf}
We need a few properties of the SHF.
Call $\{u_{\ell}\}_{\ell=1}^\infty\subset L^2(\R^2)\cap C(\R^2)$ an $L^2$-applicable mollifier if every $u_\ell$ is nonnegative and if $\phi\mapsto u_\ell\star\phi$ defines a bounded operator on $L^2(\R^2)$ that strongly converges to the identity operator as $\ell\to\infty$.
For $\mu,\nu\in M_+(\R^4)$ and nonnegative $u\in L^2(\R^2)\cap C(\R^2)$, call $(\mu\bullet_{u}\nu)$ well-defined if, for every $f\in C_c(\R^4)$, the integral
\begin{align*}
	\big( \mu \bullet_{u} \nu \big) f
	:=
	\int_{\R^{8}} \mu(\d x_1\otimes \d x_2) \, u(x_2-x_3) \, \nu(\d x_3 \otimes \d x_4) \, f(x_1,x_4)\ ,
	\quad
	x_i\in\R^2
\end{align*}
converges absolutely.
In this case, by the Riesz--Markov--Kakutani representation theorem $(\mu \bullet_{u} \nu)\in M_+(\R^4)$.
By \cite{tsai2024stochastic}, the law of $Z$ is uniquely characterized by the following Axioms.
\begin{enumerate}
\item \label{d.shf.}
$Z=Z_{s,t}$ is an $M_+(\R^4)$-valued, continuous process on $\R^2_{\leq }$.
\item \label{d.shf.ck}
For all $s< t< u$, there exists an $L^2$-applicable mollifier $\{u_\ell\}_{\ell}$ (which can depend on $s,t,u$) such that $Z_{s,t}\bullet_{u_\ell}Z_{t,u}-Z_{s,u}$ is well-defined and converges vaguely in probability to $0$ as $\ell\to\infty$.
\item \label{d.shf.inde}
For all $s< t< u$, $Z_{s,t}$ and $Z_{t,u}$ are independent.
\item \label{d.shf.mome}
For $n=1,\ldots, 4$, $s<t$, and $x=(x_1,\ldots,x_n),x'=(x'_1,\ldots,x'_n)\in\R^{2n}$,
\begin{align*}
    \EE \bigotimes_{i=1}^n Z_{s,t}(\d x_i,\d x'_i)
    =
    \d x\, \d x' \, \sg^{\intv{n}}(t-s,x,x')\ .
\end{align*}
\end{enumerate}

Here, $\sg^{\intv{n}}(t)$ denotes the 2$d$, $n$-particle delta-Bose semigroup.
It actually depends on $\theta$, which we fix.
This $\sg^{\intv{n}}(t)$ was constructed in \cite[Section~8]{gu2021moments} based on \cite{Rajeev1999,Dimock2004} as a strongly continuous semigroup on $L^2(\R^{2n})$ and it has an explicit expression.
In this paper, we will only need the expression for $n=1,2$:
\begin{align}
    \label{e.sg}
    \sg^{\intv{1}}(t)
    =
    \hk(t),
    \qquad
    \sg^{\intv{2}}(t)
    =
    \hk(t)^{\otimes 2} + \calW(t).
\end{align}
Here, $\calW(t)$ is defined by \eqref{e.calW} plugged in $\bartheta=\theta$.
Let $P_{s,t}=\d x\,\d x' \hk(t-s,x-x')$, where $x,x'\in\R^2$, and $W_{s,t}=Z_{s,t}-P_{s,t}$.
It follows from \eqref{e.sg} that
\begin{align}
    \label{e.sgW}
    \EE\, W_{s,t} = 0,
    \qquad
    \EE\, \bigotimes_{i=1}^2 W_{s,t}(\d x_i,\d x'_i)
    =
    \d x\, \d x'\, \calW(t-s,x,x'),
\end{align}
where $x=(x_1,x_2),x'=(x'_1,x'_2)\in\R^4$.

The key property that produces the log factor in \eqref{e.keyestimate} is the bound:
\begin{align}
    \label{e.jfn.bd}
    \jfn(t) := \jfn^{\theta}(t) \leq \frac{C}{t\,|\log t|^2},
    \qquad
    t \leq \tfrac{1}{2},
\end{align}
where $C$ depends on $\theta$ but not $t$.
See \cite[Equation (8.5)]{gu2021moments} for example.

It was shown in \cite[Section~2.2]{tsai2024stochastic} that, under Axioms~\eqref{d.shf.inde}--\eqref{d.shf.mome}, the limit in Axiom~\eqref{d.shf.ck} converges in $L^4$ and does not depend on the mollifier $u_\ell$, and Axiom~\eqref{d.shf.ck} extends to multifold
\begin{align}
    \label{e.ck}
    Z_{t_0,\ldots,t_k}
    =
    \lim_{\ell\to\infty}
    Z_{t_0,t_1} \bullet_{u_\ell} \ldots \bullet_{u_\ell} Z_{t_{k-1},t_{k}}
    :=
    Z_{t_0,t_1} \bullet \ldots \bullet Z_{t_{k-1},t_{k}},
\end{align}
where the limit also converges in $L^4$.
Another way of defining the product $\bullet$ was introduced in \cite{Clark2024}.
In this paper, we work with the above definition.

\subsection{Proving the variance bound \eqref{e.keyestimate}}

Fix $X$ as in \eqref{e.defX} and an interval $[s,t]$ with $t-s\leq\frac{1}{2}$. 
After reindexing, we assume that $[s,t]\subset [s_i,t_i]$ for $1\leq i\leq m$, and for $i> m$, the intersection $[s,t]\cap [s_i,t_i]$ contains at most one point.

The proof consists of two steps.

\subsubsection{Taylor expansion}

For $1\leq i\leq m$, since $[s,t]\subset[s_i,t_i]$, we use \eqref{e.ck} and $Z_{s,t}=P_{s,t}+W_{s,t}$ to write
\begin{equation}\label{e.ck1}
\begin{aligned}
    Z_{s_i,t_i}&=Z_{s_i,s}\bullet(P_{s,t}+W_{s,t})\bullet Z_{t,t_i}\\
    &= Z_{s_i,s}\bullet P_{s,t} \bullet Z_{t,t_i}+Z_{s_i,s}\bullet W_{s,t} \bullet Z_{t,t_i}
    \end{aligned}
\end{equation}
 The first term on the rhs of \eqref{e.ck1} is an $M_+(\R^4)$-valued random variable, so one may view the second term as a random signed measure.

For $i=1,\ldots,m$, define 
\[
\begin{aligned}
U_i=Z_{s_i,s}\bullet P_{s,t} \bullet Z_{t,t_i} (g_i\otimes g'_i),\quad\quad V_i=Z_{s_i,s}\bullet W_{s,t} \bullet Z_{t,t_i}( g_i\otimes g'_i),
\end{aligned}
\]
so that $U_i$ is independent of $\F_{s,t}$. Define 
\[
\F_{\mathrm{rest}}=\sigma\big(\F_{s',t'}:\,[s',t']\subset(-\infty,s]\cup[t,\infty)\big),
\]
and
\begin{align}
    &U=(U_1,\ldots,U_m), \quad V=(V_1,\ldots,V_m), 
\\
    &R=(Z_{s_{m+1},t_{m+1}}g_{m+1}\otimes g'_{m+1},\ldots, Z_{s_N,t_N}g_N\otimes g'_N),
\end{align}
so that $X=h(U+V,R)$.
Given that $h\in C_c^\infty(\R^N,\R)$, we perform a Taylor expansion in $V$:
\[
\begin{aligned}
\EE[X|\F_{s,t}]=\EE[h(U+V,R)|\F_{s,t}]=I_1+I_2+I_3,\\
\end{aligned}
\]
with 
\[
\begin{aligned}
&I_1=\EE[h(U,R)|\F_{s,t}],\quad I_2=  \EE[ \nabla h(U,R)\cdot V|\F_{s,t}],\\
& |I_3|=|\text{remainder}| \leq C  \EE [|V|^2|\F_{s,t}].
\end{aligned}
\]
Here $\nabla h$ denotes the gradient of $h$ with respect to the first $m$ variables, and the constant $C$ only depends on $h$.

Since $(U,R)$ is independent of $\F_{s,t}$, we know that $I_1=\EE[h(U,R)]$ is deterministic. For $I_2$,  write its expectation as 
\[
\EE I_2= \EE \big[\EE[\nabla  h(U,R) \cdot V|\F_{\mathrm{rest}}]\big].
\]
Since $\nabla h(U,R)$ is $\F_{\mathrm{rest}}$-measurable and $\EE[V|\F_{\mathrm{rest}}]=0$, we conclude that $\EE I_2=0$. 
As for $I_3$, by \cite[Lemma 3.1]{tsai2024stochastic}, for any $p<3/2$, there exists $C$ such that 
\[
\EE I_3^2 \leq C \EE |V|^4 \leq C  (t-s)^{p}.
\]

Putting things together gives
\[
\begin{aligned}
\EE[X|\F_{s,t}]-\EE X&=(I_1-\EE I_1)+(I_2-\EE I_2)+(I_3-\EE I_3)\\
&=I_2+(I_3-\EE I_3),
\end{aligned}
\]
so
\[
\begin{aligned}
\Var\,\EE[X|\F_{s,t}] \leq  2\EE I_2^2+2\EE I_3^2\leq 2\EE I_2^2+C(t-s)^{p}.
\end{aligned}
\]

\subsubsection{Bounding the conditional expectation}
To prove \eqref{e.keyestimate}, it remains to show that
\begin{equation}\label{e.goal1}
\EE I_2^2\leq C\frac{t-s}{|\log (t-s)|} \text{ for } t-s<\tfrac12, 
\qquad
I_2=\EE[ \nabla  F(U,R)\cdot V|\F_{s,t}].
\end{equation}
 To deal with the square of a conditional expectation, we introduce ``replicas'' that are independent of $\F_{s,t}$:  let $\tilde{Z}$ be an independent copy of $Z$, and define
\[
\begin{aligned}
\tilde{U}_i=\tilde{Z}_{s_i,s}\bullet P_{s,t} \bullet \tilde{Z}_{t,t_i} (g_i\otimes g'_i),\quad\quad \tilde{V}_i=\tilde{Z}_{s_i,s}\bullet W_{s,t} \bullet \tilde{Z}_{t,t_i}( g_i\otimes g'_i),
\end{aligned}
\]
and 
\[
\begin{aligned}
&\tilde{U}=(\tilde{U}_1,\ldots,\tilde{U}_m), \quad \tilde{V}=(\tilde{V}_1,\ldots,\tilde{V}_m), \\
&\tilde{R}=(\tilde{Z}_{s_{m+1},t_{m+1}}g_{m+1}\otimes g'_{m+1},\ldots, \tilde{Z}_{s_N,t_N}g_N\otimes g'_N).
\end{aligned}
\]
In other words, we freeze $W_{s,t}$ and replace all other random variables by their independent copies. This way, the second moment of $I_2$ can be written as
\begin{align}
    \label{e.J2}
    \EE I_2^2
    =
    \EE J_2,
    \quad
    J_2
    =
    (\nabla  h(U,R)\cdot V)  (\nabla  h(\tilde{U},\tilde{R})\cdot \tilde{V}).
\end{align}
More explicitly,
\begin{align}
    \label{e.conditional.E}
    J_2
    &=
    \sum_{i,j=1}^m \partial_i h(U,R) \, \partial_j h(\tilde{U},\tilde{R}) \, V_i\, \tilde{V}_j, 
\\
    \label{e.ViVj}
    V_i\tilde{V}_j
    &=
    Z_{s_i,s}\bullet W_{s,t} \bullet Z_{t,t_i}( g_i\otimes g'_i)
    \cdot 
    \tilde{Z}_{s_j,s}\bullet W_{s,t} \bullet \tilde{Z}_{t,t_j}( g_j\otimes g'_j).
\end{align}
Since $W_{s,t}$ is independent of every other variable in \eqref{e.conditional.E}--\eqref{e.ViVj}, we average it out first. 
Let $\EE_{W_{s,t}}$ denote the expectation on $W_{s,t}$ only.
Using the second identity in \eqref{e.sgW}, it is not hard to check that
\begin{align*}
    \EE_{W_{s,t}}&[V_i\tilde{V}_j]
    =
    \int_{\R^{16}} 
    Z_{s_i,s}(\d x_1, \d x'_1)\, \tilde{Z}_{s_j,s}(\d x_2, \d x'_2)\, Z_{t,t_i}(\d x''_1, \d x'''_1)
\\
    &
    \tilde{Z}_{t,t_j}(\d x''_2, \d x'''_2) \,
    g_i(x_1) \, g_j(x_2) \, \calW(t-s,x',x'') \, g'_i(x'''_1) \, g'_j(x'''_2).
\end{align*}
By \eqref{e.calW}--\eqref{e.jfn}, the function $\calW(t-s,x',x'')$ is positive.
With $Z_{u,v}$ and $\tilde{Z}_{u,v}$ being positive-measure-valued, we bound the last integral by taking absolute value of the $g$s:
\begin{align}
\label{e.ViVj.bd}
\begin{split}    
    |\EE_{W_{s,t}}&V_i\tilde{V}_j|
    \leq
    \int_{\R^{16}} 
    Z_{s_i,s}(\d x_1, \d x'_1)\, \tilde{Z}_{s_j,s}(\d x_2, \d x'_2)\, Z_{t,t_i}(\d x''_1, \d x'''_1)
\\
    &
    \tilde{Z}_{t,t_j}(\d x''_2, \d x'''_2) \,
    |g_i(x_1) \, g_j(x_2)| \, \calW(t-s,x',x'') \, |g'_i(x'''_1) \, g'_j(x'''_2)|.
\end{split}
\end{align}

Now, in \eqref{e.J2}, take $\EE_{W_{s,t}}$ within the $\EE$ and bound $|\nabla h|$ by a constant to get
$\
    \EE I_2^2
    \leq 
    C\sum_{i,j}\EE |\EE_{W_{s,t}}V_i\tilde{V}_j|
$.
Use \eqref{e.ViVj.bd} to bound the absolute value and use the first identity in \eqref{e.sg} to take the remaining $\EE$.
We obtain that
\begin{align*}
    \EE & I_2^2
    \leq
    C \,
    \sum_{i,j=1}^m
    \int_{\R^{16}} \d x\, 
    \hk(s-s_i,x_1-x'_1)\, \hk(s-s_j,x_2-x'_2) \, \calW(t-s,x',x'') \,
\\
    &
    \hk(t_i-t,x''_1-x'''_1)\, \hk(t_j-t,x''_2-x'''_2)
    |g_i(x_1) \, g_j(x_2) \, \, g'_i(x'''_1) \, g'_j(x'''_2)|,
\end{align*}
where $\d x = \d x_1 \cdots \d x'''_1 \, \d x_2\cdots \d x'''_2$.
Insert \eqref{e.calW} into the last integral, use the semigroup property of $\hk$ to simplify the result, and write $\hk(t)g=\int_{\R^2} \d x\, \hk(t,\cdot-x')g(x')$ to simplify notation.
We write arrive
\begin{align*}
    \EE I_2^2
    \leq
    C
    \sum_{i,j=1}^m
    &\int_{u+u'+u''=t-s} \d u \d u'\, 
    \int_{\R^2} \d y \,
    \prod_{k=i,j} \hk(u+s-s_k)|g_k| \, (y)
\\
    &\times
    \int_{\R^2} \d y' \,
    \hk(\tfrac{u'}{2},y-y')\,\jfn(u')\,
    \prod_{k=i,j} \hk(u''+t_k-t)|g'_k| \, (y').
\end{align*}
Since $g_i,g'_i\in C_c(\R^2)$, the integral over $(y,y')\in\R^4$ is bounded, uniformly over $u,u',u''\leq \frac{1}{2}$.
Hence,
\begin{align*}
    \EE I_2^2
    \leq
    C
    \int_{u+u'\leq t-s} \d u \d u'\, \jfn(u')
    \leq
    (t-s) \int_0^{t-s} \d u'\, \jfn(u').
\end{align*}
The desired bound \eqref{e.goal1} now follows from \eqref{e.jfn.bd}.

\appendix

\section{Noise sensitivity of $Z^\e$}
\label{s.noise.sensitive}
Here we prove \eqref{e.noise.sensitive}.
First, the argument for proving \cite[Equation~(3.1)]{gu2025noise} gives that
\begin{align}
    \label{e.cov}
    C'(\tau,F)
    :=
    \cov\big( F(\xi), F(e^{-\tau}\xi+\sqrt{1-e^{-2\tau}}\xi') \big)
    =
    \sum_{k \geq 1} e^{-k\tau} c_{k}(F)^2.
\end{align}
To evaluate this for $F=F^{\e}_0$, rename the $\beta_\e$ in \eqref{e.she} to $\beta$, allow $\beta$ to vary in $(0,\infty)$, and write $Z^{\e}(\xi,\beta)$ and $F_{\she}^{\e}(\xi,\beta)=Z^{\e}_{0,1}(\xi,\beta)\, g\otimes g'$ for the results.
In particular, $F_{\she}^{\e}(\xi)=F_{\she}^{\e}(\xi,\beta_\e)$.
It is straightforward to check that
\begin{align}
	\label{e.she.chaos}
	c_k\big(F_{\she}^{\e}(\cdot,\beta)\big)^2
	=
	\beta^k 
	\int \d\vec{u} \, \Ip{g^{\otimes 2}, \hk(u_0)^{\otimes 2}\Phi_\e \hk(u_1)^{\otimes 2} \, \cdots \, \Phi_\e \hk(u_k)^{\otimes 2} g'{}^{\otimes 2} },
\end{align}
where $\d\vec{u}=\d u_1\cdots\d u_k$, the integral runs over positive $u_i$s under the constraint $u_0+\ldots+u_k=1$, $\Phi_\e(x_1,x_2)=\e^{-2}\Phi(\frac{x_1-x_2}{\e})$ acts as a multiplicative operator on $\R^4$, and $\Phi$ is defined after \eqref{e.noise.mollified}.
In particular,
\begin{align}
    c_k(F_{\she}^{\e}(\cdot,\beta))^2=\beta^{k}c_k\big(F_{\she}^{\e}(\cdot,1))^2.
\end{align}
This allows us to absorb $e^{-k\tau}$ into $\beta^{k}$.
Writing $\beta_{\e,\bartau}=\beta_{\e} e^{-\bar{\tau}/2|\log\e|}$, we obtain
\begin{align*}
	C'(\tfrac{\bar{\tau}}{2|\log\e|}, F_{\she}^{\e} \big)
	&=
	\sum_{k \geq 1} c_\e\big(F_{\she}^{\e}(\cdot,\beta_{\e,\bar{\tau}})\big)^2
\\
	&=
	\EE \big(Z^{\e}_{0,1}(\cdot,\beta_{\e,\bartau})\, g\otimes g'\big)^2
	-
	\big(\EE Z^{\e}_{0,1}(\cdot,\beta_{\e,\bartau})\, g\otimes g'\big)^2\ .
\end{align*}
Note that $\EE Z^{\e}_{0,1}(\cdot,\beta)g\otimes g'=\ip{g,\hk(1)g'}$.
Dividing both sides by 
$
	\EE(F_{\she}^{\e})^2
	=
	\EE \big(Z^{\e}_{0,1}(\cdot,\beta_{\e,0})\, g\otimes g'\big)^2
	-
	\ip{g,\hk(1)g'}^2	
$ 
gives
\begin{align*}
	C(\tfrac{\bartau}{2|\log\e|}, F_{\she}^{\e} \big)
	=
	\frac{
		\EE (Z^{\e}_{0,1}(\cdot,\beta_{\e,\bartau})\, g\otimes g')^2
		-
		\ip{g,\hk(1)g'}^2
	}{
		\EE (Z^{\e}_{0,1}(\cdot,\beta_{\e,0})\, g\otimes g')^2
		-
		\ip{g,\hk(1)g'}^2
	}.
\end{align*}
Note that $\beta_{\e,\sigma}=\beta_\e\,(1-\frac{\sigma}{2}|\log\e|+O(|\log\e|^{-2}))$ is equal to
\begin{align*}
	\frac{2\pi}{|\log \eps|}+\frac{\pi}{|\log \eps|^2}\Big(\theta-\sigma-2\log 2+2\gamma+2\int_{\R^4} \d x\,\d x'\,\Phi(x)\,\log |x-x'|\,\Phi(x')\Big) 
\end{align*}
up to an error of $O(|\log\e|^{-3})$.
By \cite[Theorem~3.2]{bertini1998two}, for every $\sigma\in\R$,
\begin{align*}
	\EE \big(Z^{\e}_{0,1}(\cdot,\beta_{\e,\sigma})\, g\otimes g'\big)^2
	\longrightarrow
	\Ip{ g^{\otimes 2}\, (\hk(1)^{\otimes 2}+\calW^{\theta-\sigma}(1)) g'{}^{\otimes 2} }.
\end{align*}
Using this for $\sigma=0,\bartau$ gives the desired result.

\bibliographystyle{alpha}
\bibliography{ref.bib}

\end{document}